\documentclass[11pt]{amsart}
\usepackage[left=3.4cm,right=3.4cm]{geometry}

\usepackage[english]{babel}
\usepackage[utf8]{inputenc}

\usepackage{amsmath,amsfonts,amssymb,amscd,amsthm}
\usepackage[linktocpage,unicode]{hyperref}

\usepackage{latexsym,graphics}
\usepackage[dvips]{color}
\usepackage[dvips]{graphicx}
\usepackage{pstricks,pst-tree,rotating}

\title[Refined enumeration of noncrossing chains and hook formulas]
      {Refined enumeration of noncrossing chains and hook formulas}
\subjclass[2000]{05A18, 05E18, 11B68, 20F55.}
\date{}
\author{Matthieu Josuat-Vergès}
\keywords{noncrossing partitions, hook formulas}
\thanks{Supported by the ANR CARMA (ANR-12-BS01-0017)}
\address{CNRS and Institut Gaspard Monge, Université Paris-Est Marne-la-Vallée\\
5 Boulevard Descartes\\
Champs-sur-Marne\\
77454 Marne-la-Vallée cedex 2\\ France}
\email{matthieu.josuat-verges@univ-mlv.fr}

\hypersetup{
    unicode=true,          % non-Latin characters in Acrobat's bookmarks
    pdftoolbar=true,        % show Acrobat toolbar?
    pdfmenubar=true,        % show Acrobat menu?
    pdffitwindow=false,     % window fit to page when opened
    pdfstartview={FitH},    % fits the width of the page to the window
    pdftitle={Refined enumeration of noncrossing chains and hook formulas},
    pdfauthor={Matthieu Josuat-Vergès},
    pdfsubject={},
    pdfkeywords={},
    pdfnewwindow=true,      % links in new window
    colorlinks=true,       % false: boxed links; true: colored links
    linkcolor=red,          % color of internal links
    citecolor=blue,        % color of links to bibliography
    filecolor=magenta,      % color of file links
    urlcolor=cyan           % color of external links
}

\newtheorem{theo}{Theorem}[section]

\newtheorem{lemm}[theo]{Lemma}
\newtheorem{prop}[theo]{Proposition}

\theoremstyle{definition}
\newtheorem{defi}[theo]{Definition}
\newtheorem{rema}[theo]{Remark}

\DeclareMathOperator{\stab}{Stab}
\DeclareMathOperator{\fix}{Fix}

\begin{document}

\begin{abstract}
In the combinatorics of finite finite Coxeter groups, there is a simple formula giving the number 
of maximal chains of noncrossing partitions. It is a reinterpretation of a result by Deligne
which is due to Chapoton,
and the goal of this article is to refine the formula. First, we prove a one-parameter
generalization, by the considering enumeration of noncrossing chains
where we put a weight on some relations. Second, we consider an equivalence relation
on noncrossing chains coming from the natural action of the group on set partitions, 
and we show that each equivalence class has a simple generating function. 
Using this we recover Postnikov's hook length formula in type A and obtain a variant in type B.
\end{abstract}

\maketitle

\section{Introduction}

% The Springer number \cite{springer} $K(W)$ is such that $K(A_{n-1})$ equals 
% $T_n$, the number of alternating permutations in $\mathfrak{S}_n$.
% 
% 
% \[
%   \sum_{n\geq 0 } K(B_n) \frac{z^n}{n!} = \frac{1}{\cos(z)-\sin(z)}.
% \]
% 
% \[
%   \sum_{n\geq 0 } K(D_n) \frac{z^n}{n!} = \frac{ 2\cos(z)-1 }{ \cos(z)-\sin(z) }.
% \]

Let $W$ be a finite Coxeter group of rank $n$ and $h$ its Coxeter number.
A formula due to Deligne \cite{deligne} states that the number
of factorizations of a Coxeter element as a product of $n$ reflections is 
\[
  \frac{n!}{|W|} h^n.
\]
The value in the case of the symmetric group is $(n+1)^{n-1}$, and this number is also known
to be the number of Cayley trees on $n$ vertices.
Chapoton \cite{chapoton} give another interpretation of Deligne's formula: this number counts the 
maximal chains in the lattice of noncrossing partitions \cite{armstrong}.

Our first goal (in Section~\ref{secncc}) is to prove a one-parameter generalization of this result.
A noncrossing chain is a sequence $\hat0 = \pi_0 \lessdot \pi_1 \lessdot \dots \lessdot \pi_n = \hat 1$
in the lattice of noncrossing partitions. By weighting some of the cover relations in these chains
with a parameter $q$, the refined enumeration turns out to be 
\[
  \frac{n!}{|W|} \prod_{i=1}^n ( d_i + q(h-d_i) )
\]
where the $d_i$'s are the degrees of the group. This is done by generalizing a recursion due to Reading~\cite{reading},
and using known results on Fuss-Catalan numbers \cite{armstrong}.

Our second goal (in Section~\ref{secclassgen}) is to study the equivalence classes of noncrossing chains defined 
as follows. The group $W$ acts naturally on the set partition lattice, and there is an induced action on the
set of maximal chains of set partitions. 
The number of orbits is an integer $K(W)$ that has been calculated in our previous work \cite{josuat}.
The subset of noncrossing chains is not stable under this action, but let us say that two noncrossing chains
are equivalent if they are in the same orbit. We show that the generating function of each equivalence class 
has a simple form as a product.

Eventually (in Section~\ref{sechook}), we show how our results lead to some hook-length formula for trees in 
type A and B, more precisely, in type A we recover Postnikov's hook formula \cite{postnikov,du} and in type B we obtain a variant. 

\section*{Acknowledgement}

We thank the anonymous referee who provided the proof of Proposition~\ref{standard2}
(which in the previous version of the article was proved only for the infinite families, and for
some of the exceptional cases via a computer).

\section{Definitions}

Let $S=\{s_1,\dots,s_n\}$ be the set of simple generators of $W$, and $T$ the set of reflections.
Let $V$ be the standard geometric representation of $W$, i.e. an $n$-dimensional Euclidean space
such that each $t\in T$ is an orthogonal reflection through the hyperplane ${\rm Fix}(t) = \{ v \in V \,:\, t(v)=v \}$.
These hyperplanes are called the {\it reflecting hyperplanes}. In particular,  $H_i = {\rm Fix}(s_i)$
are called the {\it simple hyperplanes}.

\begin{defi}
 Let $\mathcal{P}(W)$ denote the set of (generalized) set partitions, i.e. linear subspaces of $V$ that are an 
 intersection of reflecting hyperplanes. It is partially ordered with reverse inclusion (i.e. $\pi\leq\rho$ 
 if $\rho \subseteq \pi$ as linear subspaces). Let $\mathcal{M}(W)$ denote the set of maximal chains of $\mathcal{P}(W)$.
\end{defi}

For each $\pi \in \mathcal{P}(W)$, we define the {\it stabilizer} and {\it pointwise stabilizer} as, respectively:
\begin{align*}
   \stab(\pi) &= \big\{ w\in W \, : \, w(L)=L \big\}, \\
   \stab^*(\pi)  &= \big\{ w\in W \, : \, \forall x \in L, \, w(x)=x  \big\}.
\end{align*}

In the classical case, an interval partition is a set partition where each block is a set of consecutive integers, 
for example $123|4|56$. %$\{\{1,2,3\},\{4\},\{5,6\}\}$.
%When we have a real reflection group $W\subset GL(V)$ together with a choice of simple generators $s_1,\dots,s_n$ and the 
%associated simple hyperplanes $H_1,\dots,H_n$, 
In the present context, there is a natural generalization (which might have been considered in previous
work, with different terminology).

\begin{defi}
 An element $\pi \in \mathcal{P}(W)$ is an {\it interval partition} if it is an intersection of simple hyperplanes.
 Let $\mathcal{P}^I(W) \subseteq \mathcal{P}(W)$ denote the set of interval partitions, and 
 $\mathcal{M}^I(W) \subset \mathcal{M}(W)$ denote the set of maximal chains in $\mathcal{P}^I(W)$.
\end{defi}

The set $\mathcal{P}^I(W)$ is a sublattice of $\mathcal{P}(W)$ and is isomorphic to a boolean lattice.
It follows that $\mathcal{M}^I(W)$ has cardinality $n!$. 
The coatoms of $\mathcal{P}^I(W)$ are exactly the lines $L_1,\dots,L_n$ defined by:
\begin{equation} \label{defli}
  L_i = \bigcap\limits_{ \substack{ 1\leq j \leq n \\[1mm] j\neq i } } H_j.
\end{equation}
Let $W_{(i)}$ denote the (standard maximal parabolic) subgroup of $W$ generated by the $s_j$ with $j \neq i$.
Then $W_{(i)} = \stab^*(L_i)$.

We will need the following fact (see \cite[Proposition~3.3]{josuat}) where $w_0$ denote the longest element
of $W$ (with respect to the simple generators $s_i$ and the associated length function).

\begin{prop} \label{wli}
Each line $L\in\mathcal{P}(W)$ can be written $w(L_i)$ for some $w\in W$ and $1\leq i \leq n$.
If $w\in W$ and $i\neq j$, then $w(L_i)=L_j$ implies $w_0(L_i)=L_j$.
\end{prop}

A consequence is the following:

\begin{prop}
Each orbit $O\in\mathcal{M}(W)/W$ contains an element of $\mathcal{M}^I(W)$.
\end{prop}

\begin{proof}
Let $C\in O$. Using Proposition~\ref{wli}, there exists $w\in W$ such that the coatom $L$ in the chain $w(C)$ is 
an interval partition, i.e. $L$ is one the $L_i$ previously defined. At this point we can make an induction
on the rank. 

Let us sketch how the induction work, using ideas present in \cite{josuat}.
There is a natural bijection between  $\mathcal{M}(W_{(i)})$ and the chains in $\mathcal{M}(W)$ having $L_i$ as 
coatom. This bijection sends $\mathcal{M}^I(W_{(i)})$ to the chains in $\mathcal{M}^I(W)$ having $L_i$ as coatom.
By induction, there is $u\in W_{(i)}$ such that $uw(C)\in\mathcal{M}^I(W)$, whence the result.
\end{proof}

Let us motivate the next definition by some considerations in the ``classical'' case.
Let $\pi_1,\pi_2,\pi_3$ be the noncrossing partitions represented in Figure~\ref{setpart} from left to right.
Here, $\pi$ is represented by drawing an arc between two consecutive elements of each block.
Both $\pi_2$ and $\pi_3$ are covered by $\pi_1$, and more precisely they are obtained from $\pi_1$ by splitting
the block $\{1,2,5,7\}$. But we can make one distinction: $\pi_2$ is obtained by removing one arc from $\pi_1$,
 and its two blocks $\{1,2\}$ and $\{5,7\}$ form an interval partition of the block $\{1,2,5,7\}$ of
$\pi_1$. This is not the case for $\pi_3$.

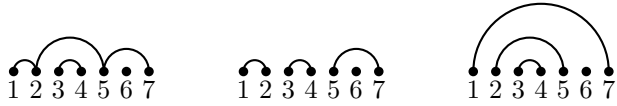
\begin{figure}[h!tp] \psset{unit=3mm}
\begin{pspicture}(1,0)(7,3)
\psdots(1,0)(2,0)(3,0)(4,0)(5,0)(6,0)(7,0)
\rput(1,-0.8){\small 1}
\rput(2,-0.8){\small 2}
\rput(3,-0.8){\small 3}
\rput(4,-0.8){\small 4}
\rput(5,-0.8){\small 5}
\rput(6,-0.8){\small 6}
\rput(7,-0.8){\small 7}
\psarc(1.5,0){0.5}{0}{180}
\psarc(3.5,0){0.5}{0}{180}
\psarc(3.5,0){1.5}{0}{180}
\psarc(6,0){1}{0}{180}
\end{pspicture}
\hspace{1cm}
\begin{pspicture}(1,0)(7,2)
\psdots(1,0)(2,0)(3,0)(4,0)(5,0)(6,0)(7,0)
\rput(1,-0.8){\small 1}
\rput(2,-0.8){\small 2}
\rput(3,-0.8){\small 3}
\rput(4,-0.8){\small 4}
\rput(5,-0.8){\small 5}
\rput(6,-0.8){\small 6}
\rput(7,-0.8){\small 7}
\psarc(1.5,0){0.5}{0}{180}
\psarc(3.5,0){0.5}{0}{180}
\psarc(6,0){1}{0}{180}
\end{pspicture}
\hspace{1cm}
\begin{pspicture}(1,0)(7,2)
\psdots(1,0)(2,0)(3,0)(4,0)(5,0)(6,0)(7,0)
\rput(1,-0.8){\small 1}
\rput(2,-0.8){\small 2}
\rput(3,-0.8){\small 3}
\rput(4,-0.8){\small 4}
\rput(5,-0.8){\small 5}
\rput(6,-0.8){\small 6}
\rput(7,-0.8){\small 7}
\psarc(3.5,0){0.5}{0}{180}
\psarc(4,0){3}{0}{180}
\psarc(3.5,0){1.5}{0}{180}
\end{pspicture}
\caption{Noncrossing partitions. \label{setpart} }
\end{figure}

To generalize this distinction, consider the group $\stab^*(\pi_1) \subset \mathfrak{S}_7$. 
It has an irreducible factor $\mathfrak{S}_4$ acting on the block $\{1,2,5,7\}$.
The simple roots of $\mathfrak{S}_7$ are $e_1-e_2,\dots, e_6-e_7$
where $(e_i)_{1\leq i \leq 7}$ is the standard basis of $\mathbb{R}^7$.
The ones of the irreducible factor $\mathfrak{S}_4$ of $\stab^*(\pi_1)$ are $e_1-e_2$, $e_2-e_5$, $e_5-e_7$.
It can be seen that the simple roots of $\stab^*(\pi_2)$ are included in the ones of $\stab^*(\pi_1)$,
but it is not the case for $\pi_3$. 

Let us turn to the general case.
Let $\Phi$ be a root system of $W$ (in the sense of Coxeter groups, see \cite{humphreys}),
and let $\Phi^+$ be a choice of positive roots. For each $\pi\in\mathcal{P}(W)$, the 
group $\stab^*(\pi)$ is a reflection subgroup of $W$, and its set of roots is $\Phi \cap \pi^{\perp}$.
We will always take $\Phi^+\cap\pi^{\perp}$ as a natural choice of positive roots, and accordingly
$\stab^*(\pi)$ has a natural set of simple roots and simple generators. In this setting, we have the following:

\begin{defi}
 Let $\pi_1,\pi_2 \in \mathcal{P}(W)$, we denote $\pi_2 \sqsubseteq \pi_1$ and say that $\pi_2$ is an 
{\it interval refinement} of $\pi_1$ if the simple roots of $\stab^*(\pi_2)$ are included in the simple 
roots of $\stab^*(\pi_1)$.
\end{defi}

Note that $\pi_2 \sqsubseteq \pi_1$ implies $\pi_1\subseteq \pi_2$, i.e. $\pi_2 \leq \pi_1$ in the lattice $\mathcal{P}(W)$.
Also, interval partitions are exactly the interval refinements of the maximal partition.

% Although the definition is valid for any pair of set partitions, here we will only use it for pairs of 
% noncrossing partitions. 

Some preliminary definitions are needed before going to noncrossing partitions.

\begin{defi}
Let $T\subset W$ be the set of reflections. A {\it reduced $T$-word} of $w$
is a factorization $w=t_1\dots t_k$ where $t_1,\dots,t_k \in T$ and $k$ is minimal.
Let $u,v\in W$, the {\it absolute order} is defined by the condition that $u <_{abs} v$
if some reduced $T$-word of $u$ is a subword of some reduced $T$-word of $v$.
\end{defi}

\begin{defi}
If $\sigma\in\mathfrak{S}_n$, 
we call $c=s_{\sigma(1)}\dots s_{\sigma(n)}$ a {\it standard Coxeter element} of $W$ with respect to $S$.
Any element conjugated in $W$ to a standard Coxeter element is called a {\it Coxeter element}. 
\end{defi}

This might differ from the terminology used in other references, but we need here some properties
of the standard Coxeter elements that are not true in general. In what follows, we always
assume that $c$ is a standard Coxeter element.

\begin{defi}
A set partition $\pi\in \mathcal{P}(W)$ is {\it noncrossing} with respect to $c$ if $\pi = \fix(w) $ for some $w\in W$
such that $w <_{abs} c$. This $w$ is actually unique and will be denoted $\underline\pi$ (see \cite[Theorem~1]{brady}).
Let $\mathcal{P}^{NC}(W,c) \subset \mathcal{P}(W) $ denote the subset of noncrossing partitions with respect to $c$, 
and $\mathcal{M}^{NC}(W,c) \subset \mathcal{M}(W)$ denote the set of maximal chains of $\mathcal{P}^{NC}(W,c)$.
If $\pi\in \mathcal{P}^{NC}(W,c)$, then $\underline \pi$ is the Coxeter element of a unique parabolic subgroup 
of $W$ that we denote $W_{(\underline\pi)}$ or $W_{(\pi)}$ (although this interferes with the notation $W_{(s)}$ 
for maximal standard parabolic subgroup, there should be no confusion).
\end{defi}

Note in particular that $\fix(\underline{\pi})=\pi$.
We refer to \cite{armstrong} for more on the subject of noncrossing partitions.
In general, $\mathcal{P}^{NC}(W,c)$ is not stable under the action of $W$.
But from the invariance of the absolute order under conjugation, we can see that 
$\mathcal{P}^{NC}(W,c)$ is stable under the action of $c$.

\begin{rema}
Noncrossing partitions are usually defined as a 
subset of $W$, but here it is natural to have the inclusion $\mathcal{P}^{NC}(W,c) \subset \mathcal{P}(W) $.
These two points of view are equivalent under the correspondence $\underline\pi \leftrightarrow \pi$ and we 
will also allow to identify noncrossing partitions with a subset of $W$. For example, if $u,v\in W$ are noncrossing, 
the notion of interval refinement $u\sqsubseteq v$ is well defined, and $u\in W$ is called an interval
partition if it is so as a noncrossing partition.
\end{rema}

\begin{prop}
We have $\mathcal{P}^I(W) \subset \mathcal{P}^{NC}(W,c)$.
Let $\pi_1\in\mathcal{P}^{NC}(W,c)$ and $\pi_2\in\mathcal{P}(W)$ with $\pi_2 \sqsubseteq \pi_1$, then $\pi_2\in\mathcal{P}^{NC}(W,c)$.
\end{prop}

\begin{proof}
The maximal partition is noncrossing since $\{0\} = \fix(c)$, so the first point follows the second one.

To prove the second point, we need Proposition~\ref{standard2} from the next section.
Let $r_1,\dots,r_k$ be the reflections associated with the simple roots of $\pi_1^{\perp}$, and we can assume there is $j\leq k$ 
such that $r_1,\dots,r_j$ are the reflections associated with the simple roots of $\pi_2^{\perp}$. Since $\pi_1$ is noncrossing,
it means there is $u\in W$ with $u<_{abs} c$ and $\fix(u)=\pi_1$. 
It is known that $u$ is a Coxeter element of the subgroup $\stab^*(\pi_1)\subset W$. 
But Proposition~\ref{standard2} shows more: it is a standard Coxeter element, so there is
$\sigma \in \mathfrak{S}_k$ such that  $u=r_{\sigma(1)} \dots r_{\sigma(k)}$. Let $v$ be obtained from this factorization by 
keeping only the factors $r_1,\dots,r_j$. Then, we have $v <_{abs} u <_{abs} c$ and $\fix(v)= \pi_2 $, so 
$\pi_2$ is noncrossing.
\end{proof}

\begin{rema}
It is interesting to note that similar results hold for {\it nonnesting partitions} in the sense of Postnikov
(defined only in the crystallographic case).
A set partition $\pi \in \mathcal{P}(W)$ is nonnesting when the simple roots of $\stab^*(\pi)$ form an antichain
in the poset of positive roots. A subset of an antichain being itself an antichain,  if $\pi_2 \sqsubseteq \pi_1 $
and $\pi_1$ is nonnesting, then $\pi_2$ is nonnesting. Any interval partition is nonnesting, since the simple
roots form an antichain. Note also that the intuition from the ``classical'' case is clear: it is impossible
to create a crossing or a nesting by removing arcs.
\end{rema}

% Let $\pi\in\mathcal{P}^{NC}(W,c)$. Then, $\pi\in\mathcal{P}^{I}(W,c)$ iff $\ell(\pi) = \ell_T(\pi) $.
% 

%%%%%%%%%%%%%%%%%%%%%%%%%%%%%%%%%%%%%%%%%%
\section{Chains of noncrossing partitions}
%%%%%%%%%%%%%%%%%%%%%%%%%%%%%%%%%%%%%%%%%%
\label{secncc}

\begin{defi}
For any chain $\Pi=(\pi_0,\dots,\pi_n)\in\mathcal{M}^{NC}(W,c)$, let 
${\rm nir}(\Pi)$ be the number of $i$ such that $\pi_i$ is not an interval refinement of $\pi_{i+1}$.
Let 
\[
  M(W,q) = \sum_{ \Pi \in \mathcal{M}^{NC}(W,c) } q^{ {\rm nir} (\Pi) }.
\]
\end{defi}

It is not {\it a priori} obvious that $M(W,q)$ does not depend on the choice of the standard Coxeter element $c$.
%This will follow some particular properties of the standard Coxeter elements.
This will be proved below.

The coatoms of the lattice $\mathcal{P}^{NC}(W,c)$ are exactly the products $ct$ for $t\in T$.
Since $T$ is stable by conjugation, the set $cT$ of coatoms is stable by conjugation by $c$.
An interesting property of standard Coxeter elements is that this action has good properties, 
(see Propositions~\ref{standard1} and \ref{standard2})
similar to those of a bipartite Coxeter element obtained by Steinberg \cite{steinberg}.

In what follows, an orbit for the action of $c$ will be called a {\it $c$-orbit}. Note that 
the action of $c$ becomes conjugation when we see noncrossing partitions as elements of $W$,
i.e. $\underline{ c(\pi) } = c \underline\pi c^{-1}$ if $\pi\in\mathcal{P}^{NC}(W,c)$.

\begin{prop}   \label{standard1}
%Let $(W,S)$ be an irreducible Coxeter system and $c$ a standard Coxeter element. 
Let $h$ be the Coxeter number of $W$ (i.e. the order of $c$ in $W$).
For any $t\in T$, the $c$-orbit of $ct$ satisfies one of the following condition:
\begin{itemize}
 \item It contains $h$ distinct elements, and exactly 2 interval partitions $L_i$ and $L_j$, related by $L_i=w_0(L_j)$.
 \item Or it contains $\frac h2$ distinct elements, and exactly 1 interval partition $L_i$, satisfying $w_0(L_i)=L_i$.
       Moreover, $c^{h/2}$ restricted to $L_i$ is $-1$ (i.e. $c^{h/2} \notin W_{(i)}$).
\end{itemize}
% Moreover, two coatoms $L_i \neq L_j$ in $\mathcal{P}^I(W)$ are in the same $c$-orbit if and only if $w_0(L_i)=L_j$.
% And if the interval partition $L_i$ is such that $c^{\frac h2}(L_i)=L_i $, then we have that $c^{\frac h2}$ restricted to $L_i$ is $-1$.
\end{prop}

The full proof is in Appendix~\ref{appen} but let us give some comments.
A standard Coxeter element $c=s_{\sigma(1)}\dots s_{\sigma(n)}$ is called {\it bipartite} if 
there is $j$ such that $s_{\sigma(1)} , \dots , s_{\sigma(j)}$ are pairwise commuting, 
and $s_{\sigma(j+1)}\dots s_{\sigma(n)}$ too. Steinberg \cite{steinberg} proved
that for a bipartite Coxeter element $c$, the $c$-orbit of a reflection contains either
$h$ elements and 2 simple reflections, or $\frac h2$ elements and 1 simple reflection.
If $h$ is even, another property of the bipartite Coxeter element is $c^{h/2} = w_0 $.
What we have is a variant that holds for any standard Coxeter element.
It is natural to expect that our result can be seen as a consequence of Steinberg's but
we have been unable to realize this in a uniform way.

Since the standard Coxeter element $c$ is conjugated with a bipartite Coxeter element, 
and the bijection $t\mapsto ct$ from $T$ to $cT$ commutes with $c$-conjugation, we see that
the $c$-orbit of $ct$ contains either $h$ or $\frac h2$ elements.
In the case where $w_0$ is central, we can easily complete the proof of Proposition~\ref{standard1}.
It is known that in this case, $h$ is even and $c^{h/2} = w_0 = -1 $, which acts trivially 
on $\mathcal{P}(W)$ (see \cite[Section 3.19]{humphreys}). So every orbit has $\frac h2$ elements.
Proposition~\ref{wli} shows that there is at most one interval partition in each orbit, 
and the equality $\#T = \frac {nh}2 $ shows that there is exactly one interval partition in each orbit.
See Appendix~\ref{appen} for the other cases.%, as well as for the proof of proposition below.

\begin{rema} \label{permutefactor}
Suppose $h$ is even and let $L_i$ be such that $c^{h/2}(L_i)=L_i$. 
As mentioned above, we have $c^{h/2}=w_0$ when $c$ is a bipartite Coxeter element.
In the general case, since $w_0$ and $c^{h/2}$ are both in $\stab(L_i) - \stab^*(L_i)$, we have
$w_0c^{h/2} \in W_{(i)}$. From the properties of $x\mapsto w_0 x w_0$, one can deduce
that the map $x\mapsto c^{h/2} x c^{h/2}$ permutes the irreducible factors of $W_{(i)}$ in the same
way as $x\mapsto w_0 x w_0$. This will be needed in the sequel.
\end{rema}

It is known that parabolic Coxeter elements can be characterized with the absolute order, see \cite[Lemma 1.4.3]{bessis},
so that $ct$ is a Coxeter element of $W_{(ct)}$. The point of the next proposition is that it is actually a standard 
Coxeter element.

\begin{prop}  \label{standard2}
For any $t\in T$, $ct$ is a standard Coxeter element of the parabolic subgroup $W_{(ct)}$ for the natural 
choice of simple generators.
\end{prop}

\begin{proof}
The elements $ct$ ($t\in T$) are the coatoms of $\mathcal{P}^{NC}(W)$.
By an immediate induction, the proposition implies (and therefore is equivalent) to the 
stronger fact that $\underline\pi$ is a standard Coxeter element of $\stab^*(\pi)$
for each $\pi\in\mathcal{P}^{NC}(W)$.
The proof of this has been provided by an anonymous referee, and relies on results by Reading~\cite{reading2}.

More specifically, the result follows from \cite[Theorem~6.1]{reading2}. A consequence of this theorem is that a 
noncrossing partition $\underline \pi$ is a product of its so-called {\it cover reflections}. 
Besides, \cite[Lemma~1.3]{reading2} states that these cover reflections are the simple generators of a parabolic 
subgroups.
\end{proof}

We are now ready to prove how $M(W,q)$ can be computed inductively, and in particular that it does not depend
on the choice of a standard Coxeter element.

\begin{prop}  \label{proprecmwq}
If $W$ is irreducible, we have:
\begin{equation} \label{recmwq}
   M(W,q) = \frac{2+q(h-2)}{2} \sum_{s\in S} M( W_{(s)} , q ).
\end{equation}
\end{prop}

\begin{proof}
For each $\Pi = (\pi_0,\dots,\pi_n) \in \mathcal{M}^{NC}(W,c)$, let $\Pi' = (\pi_0,\dots,\pi_{n-1}) $.
The coatom of $\Pi$ is $\pi_{n-1}=ct$ for some $t\in T$, and the set of such $\Pi$ with $ct$ as coatom 
is in bijection with $\mathcal{M}^{NC}(W_{(ct)},ct)$ via the map $\Pi \mapsto \Pi'$. Moreover,
${\rm nir}(\Pi) = {\rm nir}(\Pi') $ if $ct \sqsubseteq c $ (i.e. $ct\in \mathcal{P}^I(W)$) 
and ${\rm nir}(\Pi) = {\rm nir}(\Pi') + 1 $ otherwise.
So, distinguishing the chains in $\mathcal{M}^{NC}(W,c)$ according to their coatoms gives:
\begin{equation} \label{ind1}
  M(W,q) = \sum_{t\in T} q^{ \chi [ \; ct \notin \mathcal{P}^I(W) \; ]  }  M( W_{(ct)} , q ).
\end{equation}
%In this sum, we gather all $t\in T$ that are an orbit under conjugation by $c$.
Note that to write this equation, we need to use Proposition~\ref{standard2}. 
While it should be clear from the definition that the generating function of the chains
$ (\pi_0,\dots,\pi_{n-1}) \in \mathcal{M}^{NC}(W_{(ct)},ct) $ with respect to the statistic ${\rm nir}$
is $M( W_{(ct)} , q )$, this quantity was only defined with respect to a standard Coxeter element.
Since $ct$ is indeed a standard Coxeter element of $W_{(ct)}$, we get the term $M( W_{(ct)} , q )$
which we assume we already know by induction.

Let $O\subset T$ be an orbit under conjugation by $c$.
 So
if $t_1,t_2\in O$, $W_{(ct_1)}$ and $W_{(ct_2)}$ are conjugated in $W$, so they are isomorphic
and $M( W_{(ct_1)} , q )=M( W_{(ct_2)} , q )$.
If $cO = \{ co \; : \; o\in O \}$ contains $h/2$ elements and 1 interval partition $L_i$, we get
\begin{equation} \label{ind2}
  \sum_{t\in O} q^{ \chi [ \; ct \notin \mathcal{P}^I(W) \; ]  }  M( W_{(ct)} , q ) = (1+q(\tfrac h2 - 1) ) M(W_{(i)},q).
\end{equation}
If it contains $h$ elements and 2 interval partitions $L_i$ and $L_j$, then 
\[
  \sum_{t\in O} q^{ \chi [ \; ct \notin \mathcal{P}^I(W) \; ]  }  M( W_{(ct)} , q ) = (2+q(h - 2) ) M(W_{(i)},q),
\]
and since the previous equation is true with $i$ replaced with $j$, we also have
\begin{equation} \label{ind3}
  \sum_{t\in O} q^{ \chi [ \; ct \notin \mathcal{P}^I(W) \; ]  }  M( W_{(ct)} , q ) = \frac{2+q(h - 2)}2 ( M(W_{(i)},q) + M(W_{(j)},q) ).
\end{equation} 
Now, we can split the sum in the righ-hand side of \eqref{ind1} to group together the $t\in T$ that are in the same orbit, 
and using Equations~\eqref{ind2} and \eqref{ind3},  we get the desired formula for $M(W,q)$.
\end{proof}

Besides, in the reducible case it is straightforward to show that
\begin{equation} \label{recmwq2}
  M(W_1\times W_2,q) = \binom{m+n}{m} M(W_1,q)\times M(W_2,q)
\end{equation}
if the respective ranks of $W_1$ and $W_2$ are $m$ and $n$.

Equation~\eqref{recmwq} and \eqref{recmwq2} can be used to compute $M(W,q)$ by induction for any $W$, with the initial
value $M(A_1,q)=1$.

This recursion permits to make a link with the Fuss-Catalan numbers ${\rm Cat}^{(m)}(W)$ (see \cite[Chapter 5]{armstrong}).
These numbers can be defined in terms of the degrees of the group $d_1,\dots,d_n$ and the Coxeter number $h=d_n$ by 
\[
  {\rm Cat}^{(m)}(W) = \frac{1}{|W|} \prod_{i=1}^n (hm + d_i).
\]
Chapoton \cite{chapoton} showed that ${\rm Cat}^{(m)}(W)$ is the number of multichains $\pi_1\leq\dots \leq \pi_m$ in $\mathcal{P}^{NC}(W,c)$,
i.e. ${\rm Cat}^{(m)}(W)=Z(W,m+1)$ where $Z(W,m)$ is the zeta polynomial of $\mathcal{P}^{NC}(W,c)$.
Fomin and Reading \cite{fomin} introduced the so-called generalized cluster complex $\Delta^m(\Phi)$, and showed that its number 
of maximal simplices is ${\rm Cat}^{(m)}(W)$ (where $\Phi$ is the root system of $W$).
Using this generalized cluster complex, they obtain in \cite[Proposition 8.3]{fomin} that
\begin{equation} \label{recfomin}
  {\rm Cat}^{(m)}(W) = \frac{(m-1)h+2}{2n} \sum_{s\in S} {\rm Cat}^{(m)}(W_{(s)})
\end{equation}
in the irreducible case. Besides, there holds
\begin{equation} \label{recfomin2}
  {\rm Cat}^{(m)}(W_1\times W_2) = {\rm Cat}^{(m)}(W_1) \times {\rm Cat}^{(m)}(W_2)
\end{equation}
in the reducible case.
Comparing the recursions \eqref{recmwq}, \eqref{recmwq2} and \eqref{recfomin}, \eqref{recfomin2} shows that
\[
  M(W,q) = n! (1-q)^n Z\big(W,\tfrac{1}{1-q}\big),
\]
where we use the zeta polynomial rather than writing ``${\rm Cat}^{(\frac{q}{1-q})}(W)$''
because it is generally assumed that $m\in\mathbb{N}$ when we write ${\rm Cat}^{(m)}(W)$.
Then, the formula for ${\rm Cat}^{(m)}(W)$ in terms of the degrees proves the proposition below
(note that the particular case $q=1$ is the result by Chapoton mentioned in the introduction):

\begin{prop}
\[
  M(W,q) = \frac{n!}{|W|}  \prod_{i=1}^n \big( d_i + q(h-d_i) \big).
\]
\end{prop}

It is also possible to obtain this formula by solving the recursion \eqref{recmwq} case by case.
We will not give the details, since lengthy calculations are needed for the differential equations 
arising in the infinite families case.
Let us just present the case of the group $A_n$, where we get that the 
series $A(z)=\sum_{n\geq 0} M(A_n,q) \frac{z^n}{n!}$ satisfies the differential equation
\[
  A' = A^2 + \tfrac{qz}{2} (A^2)'.
\]
After multiplying the equation by $A^{q-2}$, it can be rewritten
\[
  \bigg(\frac{A^{q-1}}{q-1}\bigg)' = (zA^q)'.
\]
After checking the constant term, we arrive at the functional equation $A^{q-1} = 1 + (q-1)zA^q$.
It would be possible to extract the coefficients of $A$ with the Lagrange inversion formula.
Another method is to use results about Fuss-Catalan numbers in type A.
It is known that ${\rm Cat}^{(m-1)}(A_{n-1}) = \frac{1}{mn+1}\binom{mn+1}{n}$, which 
is the number of complete $m$-ary trees with $n$ internal vertices, so that 
$F = 1+ \sum_{n\geq 1} {\rm Cat}^{(m-1)}(A_{n-1}) z^n $ satisfies $F=1+zF^m$.
The equation for $A$ can be rewritten
\[
  A^{1-q} = 1 + z(1-q)A
\]
So, comparing the functional equations shows $F(z) = A(\frac{z}{1-q} )^{1-q} $ if $m=\frac{1}{1-q}$.
This is also $F(z)=1+zA(\frac{z}{1-q})$.
Taking the coefficient of $z^{n+1}$, we obtain: 
\[
  \frac{ 1 }{ \frac{n+1}{1-q} +1 } \binom{ \frac{n+1}{1-q}+1 }{ n+1 } =   \frac{1}{(1-q)^n n!} M(A_n,q),
\]
hence
\[
  M(A_n,q) = \frac{n! (1-q)^n}{ \frac{n+1}{1-q} +1 } \binom{ \frac{n+1}{1-q}+1 }{ n+1 }  = \prod_{i=1}^{n-1} ( i+1  + q(n-i)  ).
\]
\section{Generating functions of equivalence classes and hook formulas.}
%%%%%%%%%%%%%%%%%%%%%%%%%%%%%%%%%%%%%%%%%%%%%%%%%%%%%%%%%%%%%%%%%%%%%%%%
\label{secclassgen}

% In this section, we will need another property of standard Coxeter elements, proved in Appendix~\ref{appen}.
% 
% \begin{prop} \label{standard3}
% Let $L_i$ be a coatom of $\mathcal{P}^I(W)$, and suppose that $ \stab^*(L_i) \subsetneq \stab(L_i) $. Then, 
% the Coxeter number $h$ is even, and $c^{\frac h2} \in \stab(L_i) - \stab^*(L_i) $.
% \end{prop}
% 
% Note that from Proposition~\ref{propliwi}, in the case above we have $w_0 \in \stab(L_i) - \stab^*(L_i)$.
% In general, if $h$ is even, the bipartite Coxeter element is such that $c^{\frac h2} = w_0 $ (see \cite[Section 3.19]{humphreys}).
% We can see the previous proposition as a (rather weak) link between $w_0$ and $c^{\frac h2}$
% in the non bipartite case.

\begin{defi}
For any $\Pi \in \mathcal{M}^{NC}(W,c)$, let $[\Pi]$ denote its equivalence class for the $W$-action:
\[ 
  [\Pi] = \{ w(\Pi) \; : \; w\in W \} \cap \mathcal{M}^{NC}(W,c).
\]
We also define the class generating function:
\[
  M([\Pi],q) = \sum_{\Omega\in [\Pi]} q^{{\rm nir}(\Omega)}.
\] 
\end{defi}

% Note that if we see noncrossing partitions as a subset of $W$, the definition of the classes would
% not be quite as simple and natural.
These classes partition the set $\mathcal{M}^{NC}(W,c)$, so that we have
\begin{equation} \label{classeq}
  M(W,q) = \sum_{[\Pi]}  M([\Pi],q)
\end{equation}
where we sum over all distinct equivalence classes.
%In the previous section, we have obtained the generating function $M(W,q)$ in the form of a product.
% We can also obtain the generating function of each equivalence class in the form of a simple product,
% and we will see how the previous equation can be interpreted as a hook-length formula in type A and B.

We need some definitions before giving the formula for $M([\Pi],q)$.

Let $\tau \lessdot \pi$ be a cover relation in $\mathcal{P}^{NC}(W,c)$.
The group $W_{(\pi)}$ can be decomposed into irreducible factors (that can be thought of as ``blocks'' of the 
set partition $\pi$). There is only one of these factors where $\underline \tau$ and $ \underline \pi$ differ, as can be seen
from the factorization of the poset $\mathcal{P}(W_{(\pi)})$ induced by the factorization of $W_{(\pi)}$.

\begin{defi}
With $\tau$ and $\pi$ as above, let $h(\tau,\pi)$ be the Coxeter number of the irreducible factor of $W_{(\pi)} $
where $\underline\tau$ and $\underline\pi$ differ. 
\end{defi}

\begin{defi}
%With $\tau$ and $\pi$ as above, there is a coatom $L\in\mathcal{P}(W_{(\pi)})$ such that $\tau=\fix(L)$.
Let $g(\tau,\pi)$ be minimal $g>0$ such that $\underline\pi^g \, \underline \tau \, \underline \pi ^{-g} = \underline\tau $
and the map $x\to \underline\pi^g x \underline\pi^{-g}$ stabilizes each irreducible factor of $W_{(\tau)}$.
\end{defi}

% There is a coatom $L\in\mathcal{P}^{NC}(W_{(\pi)},\underline\pi)$ such that $L=\fix(\tau)$, and $W_{(\tau)}=\stab^*(L) $. 
% Consider the minimal $k>0$ such that $\underline\pi^k(L)=L$. In particular, $k$ divides $g(\tau,\pi)$.
% From the properties of standard Coxeter elements, we have either $k=h(\tau,\pi)$ or $k=\frac 12 h(\tau,\pi)$.
% By examining the irreducible factors of $W_{(\pi)}$, we can see that $\underline \pi ^{ h(\tau,\pi) } \in W_{(\tau)} $,
% so that $g(\tau,\pi)$ divides $h(\tau,\pi)$.

Note that by examining the irreducible factors of $W_{(\pi)}$, we can see that we have
$\underline\pi^{h(\tau,\pi)} \, \underline \tau \, \underline \pi ^{-h(\tau,\pi)} = \underline\tau $.
From $\underline\pi^g \, \underline \tau \, \underline \pi ^{-g} = \underline\tau $ and Proposition~\ref{standard2}, 
we have either $g(\tau,\pi) = h(\tau,\pi)$ or $g(\tau,\pi) = \frac12 h(\tau,\pi)$.
Note also that when $h(\tau,\pi)$ is even, as noted in Remark~\ref{permutefactor}, we known that the map 
$x\to \underline\pi^{\frac 12 h(\tau,\pi)} x \underline\pi^{-\frac 12 h(\tau,\pi)}$ permutes the irreducible factors
of $W_{(\tau)}$.

\begin{prop} \label{classgen}
Let $ \Pi = (\pi_0,\dots,\pi_n) \in\mathcal{M}^{NC}(W,c) $, let $h_i=h(\pi_{i-1},\pi_i)$ and $g_i=g(\pi_{i-1},\pi_i)$
for $2\leq i \leq n$. Then we have:
\[
  M( [\Pi] , q ) = \prod_{i=2}^{n} \Big( \frac{2g_i}{h_i} + q \Big(g_i - \frac{2g_i}{h_i}  \Big)  \Big).
\]
\end{prop}

The proof is rather similar to that of Proposition~\ref{proprecmwq}.
We need a few lemmas.

\begin{lemm} \label{classlem1}
If $\Omega = (\omega_0,\dots,\omega_n ) \in [\Pi]$, there is $k\geq 0$ such that $\omega_{n-1} = c^k (\pi_{n-1} )$.
\end{lemm}

\begin{proof}
Let $L_i$ (respectively, $L_j$) be an interval partition in the $c$-orbit of $\omega_{n-1}$ (respectively, $\pi_{n-1}$).
The fact that these exist follows Proposition~\ref{standard1}.
If $L_i=L_j$, the $c$-orbits are the same and this ends the proof. 

Suppose now that $L_i\neq L_j$. Since $\Omega\in[\Pi]$, there is $w\in W$ such that $w(L_i)=L_j$, so
Proposition~\ref{wli} shows that $w_0(L_i)=L_j$. Then, Proposition~\ref{standard1} shows that $L_i$ and $L_j$
are in the same $c$-orbit. So $\omega_{n-1}$ and $\pi_{n-1}$ are in the same $c$-orbit.
\end{proof}

\begin{lemm}
Let $\Omega = (\omega_0,\dots,\omega_n ) \in[\Pi]$, and
assume inductively that Proposition~\ref{classgen} is true for the group $W_{(\omega_{n-1})}$.
Let $\langle \Omega \rangle$ denote the class of $\Omega$ for the action of $W_{(\omega_{n-1})}$, i.e.
\[
  \langle \Omega \rangle = \{ w(\Omega) \; : \; w\in W_{(\omega_{n-1})} \} \cap \mathcal{M}^{NC}(W,c).
\]
Then the generating function of $\langle\Omega\rangle$ is:
\begin{equation} \label{genomega}
  M( \langle\Omega\rangle ,q) = q^{\chi[ \;  \omega_{n-1} \notin \mathcal{P}^I(W)  \;  ]} 
                                \prod_{i=2}^{n-1} \Big( \frac{2g_i}{h_i} + q \Big(g_i - \frac{2g_i}{h_i}  \Big)  \Big).
\end{equation}
\end{lemm}

\begin{proof}
Let $\Omega'=(\omega_0,\dots,\omega_{n-1})$. Removing the last element of a chain gives a bijection between $\langle \Omega \rangle$ and
\[
  [\Omega'] =  \{ w( \Omega' ) \, : \, w \in W_{(\omega_{n-1})} \} \cap \mathcal{M}^{NC}( W_{(\omega_{n-1})} , \underline{\omega_{n-1}} ).
\]
By induction, we can obtain $M([\Omega'],q)$. Since $\Omega \in [\Pi]$, it is straightforward to check that we have 
$g(\omega_{i-1},\omega_i)=g(\pi_{i-1},\pi_i)$ and $h(\omega_{i-1},\omega_i)=h(\pi_{i-1},\pi_i)$, 
although we see $\omega_{i-1},\omega_i$ as elements of $\mathcal{P}^{NC}(W_{(\omega_{n-1})}, \underline{\omega_{n-1}})$
and $\pi_{i-1},\pi_i$ as elements of $\mathcal{P}(W,c)$.
We have $M(\langle\Omega\rangle,q) = q^{\chi[ \;  \omega_{n-1} \notin \mathcal{P}^I(W)  \;  ]}  M([\Omega'],q)$,
and this gives the formula for $M(\langle\Omega\rangle,q)$.
\end{proof}

\begin{lemm}
The minimal integer $g>0$ such that $\langle \Pi \rangle = \langle c^g(\Pi) \rangle$ is $g_n$.
\end{lemm}

\begin{proof}
This $g$ satisfies $c^g(\pi_{n-1})=\pi_{n-1}$, so that either $g=h_n$ or $g=\frac{h_n}2$.
If we are not in the case where $c^{h_n/2}(\pi_{n-1})=\pi_{n-1}$, we have $g=h_n=g_n$.
So, suppose $c^{h_n/2}(\pi_{n-1})=\pi_{n-1}$.

Consider the factorization of the poset $\mathcal{P}(W_{(\pi_{n-1})})$ induced by the factorization of $W_{(\pi_{n-1})}$
in irreducible factors. From the definition of $g_n$, the action of $c^{g_n}$ stabilizes each factor
of the poset, so it is the same action as some element $w\in W_{(\pi_{n-1})}$. 
So $\langle \Pi \rangle = \langle c^{g_n}(\Pi) \rangle$ and this proves $g\leq g_n$.

Reciprocally, suppose that $c^g(\Pi)=w(\Pi)$ for some $w\in W_{(\pi_{n-1})}$.
It follows that $c^g$ stabilizes the irreducible factors of $W_{(\pi_{n-1})}$.
%The argument is similar to the one in the third part of Proposition~\ref{invohalf}:
If the permutation on the factors is nontrivial, it would be possible to distinguish
$c^g(\Pi)$ from $w(\Pi)$. So $g_n\geq g$, and eventually $g=g_n$. 
\end{proof}

\begin{lemm}
% Let $0\leq k < h_n$. If $g_n=\frac{h_n}2$, then the set of $\Omega\in [\Pi]$ such that $\omega_{n-1}=c^k(\pi_{n-1})$ form
% a single class $\langle \Omega \rangle$.
% If $g_n=h_n$, then this set is the union of two distinct classes $\langle \Omega^{(1)} \rangle$ and $\langle \Omega^{(2)} \rangle$.
The classes $\langle \Omega \rangle$ form a partition of the set $[\Pi]$.
A set of representatives is $\{\Pi,c(\Pi),\dots,c^{g_n-1}(\Pi)\}$.
\end{lemm}

\begin{proof}
The first point is clear.
From the previous lemma, the elements in the set $\{\Pi,c(\Pi),\dots,c^{g_n-1}(\Pi)\}$ are in distinct classes. 
It remains to show that the list is exhaustive.

Knowing Lemma~\ref{classlem1}, it remains to prove that if $\Omega\in[\Pi]$ is such that $\omega_{n-1}=\pi_{n-1}$,
then there is $k$ such that $ \langle \Omega \rangle = \langle c^k(\Pi) \rangle $.
Let $w\in W$ such that $\Omega = w(\Pi)$. In particular, $w(\pi_{n-1})=\pi_{n-1}$.

If $w\in W_{(\pi_{n-1})}$, we have $\langle \Omega \rangle = \langle \Pi \rangle $.
Otherwise, it means that $w\in \stab(\pi_{n-1})-\stab^*(\pi_{n-1}) $.
Since the class $[\Pi]$ contains a chain of interval partitions, we might as well assume that $\pi_{n-1}$
is an interval partition. 
It comes from Proposition~\ref{standard1} that $w c^{h/2} \in W_{(\pi_{n-1})} $.
So we obtain $ \langle \Omega \rangle = \langle c^{h/2}(\Pi) \rangle $.
This completes the proof.
\end{proof}

We can now prove Proposition~\ref{classgen}.

\begin{proof}
Since the classes $\langle \Omega \rangle$ form a partition of $[\Pi]$, we have:
\[
  M([\Pi],q) = \sum_{\langle \Omega \rangle } M(\langle \Omega \rangle ,q ),
\]
and $M([\Pi],q)$ can be obtained by summing Equation~\eqref{genomega}.

From the previous lemma, the number of distinct classes $\langle\Omega\rangle$ is $g_n$.
As we have seen above (just before Proposition~\ref{classgen}), either $g_n=h_n$ or $g_n= \frac 12 h_n$,
so that $\frac{2g_n}{h_n} $ is an integer.
From Proposition~\ref{standard1}, $\frac{2g_n}{h_n}$ among the distinct classes $\langle\Omega\rangle$ 
are such that their coatom is an interval partition. So, we get 
\[
  \sum_{\langle \Omega \rangle }  q^{\chi[ \;  \omega_{n-1} \notin \mathcal{P}^I(W)  \;  ]}  = 
     \Big( \frac{2g_n}{h_n} + q \Big(g_n - \frac{2g_n}{h_n}  \Big)  \Big).
\]
So, we get the desired formula for $M([\Pi],q)$ by
summing Equation~\eqref{genomega} over the classes $\langle \Omega \rangle $.
\end{proof}

% \begin{rema} \label{defK}
% Let $\Pi = (w_0 , \dots , w_n ) \in \mathcal{M}^{NC}(W,c)$, and here we suppose $w_i\in W$ (we see noncrossing partitions as a subset 
% of $W$). For each $i$ such that $1\leq i \leq n $, we can create a new chain $\delta_i(\Pi)$:
% \[
%    (\pi_i \pi_0 \pi_i^{-1},\dots, \pi_i \pi_{i-1} \pi_i^{-1}, \pi_i ,\dots ,\pi_n).
% \]
% The class $[\Pi]$ is obtained from $\Pi$ by taking all the possible images under some sequences of these operations $\delta_i$.
% This is a consequence of what we have seen above, 
% 
% ++++++++++++++++++++++++++++++++++++++
% 
% 
% The set $\mathcal{M}^{NC}(W,c)$ is in bijection with factorization of $c$ as a product of $n$ reflections.
% Explicitly, from a factorization $c=t_1\cdots t_n$, we form the chain $\epsilon <_{abs} t_1 <_{abs} t_1t_2 <_{abs}\dots <_{abs} c $
% where $\epsilon$ is the neutral element.
% \end{rema}

%%%%%%%%%%%%%%%%%%%%%%%%%%%%%%%%%%%%%%%%%
\section{Hook formulas for types A and B}
%%%%%%%%%%%%%%%%%%%%%%%%%%%%%%%%%%%%%%%%%
\label{sechook}

This section is devoted to explicit combinatorial description in type A and B, where
Equation~\eqref{classeq} can be interpreted as a hook-length formula for trees.

\begin{defi} \label{defandretrees}
  Let $\mathcal{A}_n$ denote the set of {\it André trees} on $n$ vertices, i.e. trees such that:
\begin{itemize}
 \item each internal node has either one son or two unordered sons,
 \item the vertices are labeled with integers from $1$ to $n$, and the labels are decreasing from the root to the leaves.
\end{itemize}
%unary-binary decreasingly-labeled unordered trees
\end{defi}

The 5 elements of $\mathcal{A}_4$ are represented in Figure~\ref{tree5}.
These trees were introduced by Foata and Schützenberger \cite[Chapter 5]{foata1}, who proved that $\# \mathcal{A}_n = T_n$
(in fact their definition requires increasing labels instead of decreasing here, but this is clearly equivalent).
They were also used by Stanley \cite{stanley2} to prove $K(A_n)=T_n$. 

\begin{figure}[h!tp]
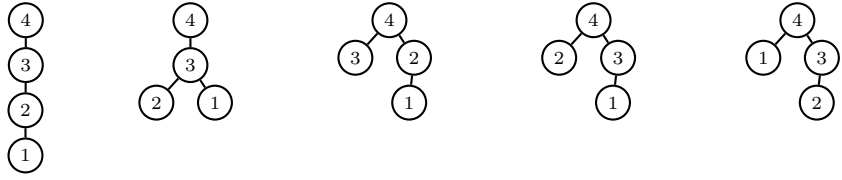

\pstree[levelsep=6mm]
       { \Tcircle{\tiny 4} }
       { 
         \pstree[levelsep=6mm]  { \Tcircle{\tiny 3} }  { 
                                                                  \pstree[levelsep=6mm]  { \Tcircle{\tiny 2} }  { \Tcircle{\tiny 1} } 
                                                        } 
       }
\hspace{1cm}
\pstree[levelsep=6mm]
       { \Tcircle{\tiny 4} }
       { 
         \pstree[levelsep=5mm,treesep=3mm]  { \Tcircle{\tiny 3} }  {  { \Tcircle{\tiny 2} } { \Tcircle{\tiny 1} }  } }
\hspace{1cm}
\pstree[levelsep=5mm,treesep=3mm]
       { \Tcircle{\tiny 4} }
       { 
          { \Tcircle{\tiny 3}  }
          {\pstree[levelsep=6mm,treesep=3mm]  { \Tcircle{\tiny 2} }  { 
                                                           { \Tcircle{\tiny 1} }
                                                        } 
          }
       }
\hspace{1cm}
\pstree[levelsep=5mm,treesep=3mm]
       { \Tcircle{\tiny 4} }
       { 
          { \Tcircle{\tiny 2}  }
          {\pstree[levelsep=6mm,treesep=3mm]  { \Tcircle{\tiny 3} }  { 
                                                           { \Tcircle{\tiny 1} }
                                                        } 
          }
       }
\hspace{1cm}
\pstree[levelsep=5mm,treesep=3mm]
       { \Tcircle{\tiny 4} }
       { 
          { \Tcircle{\tiny 1}  }
          {\pstree[levelsep=6mm,treesep=3mm]  { \Tcircle{\tiny 3} }  { 
                                                           { \Tcircle{\tiny 2} }
                                                        } 
          }
       }
\caption{The André trees with 4 vertices. \label{tree5}}
\end{figure}

Let us describe Stanley's bijection. We see it as a map $\mathcal{M}(A_{n-1}) \to \mathcal{A}_n$
that induces a bijection $\mathcal{M}(A_{n-1}) / A_{n-1}  \to \mathcal{A}_n $.
We present an example on Figure~\ref{stanmap} and refer to \cite{stanley2} for more details.
Suppose that we start from the minimal partition $1|2|3|4|5|6|7$ and at each step,
two blocks merge into a larger block. We need 6 steps before arriving to the maximal partition $1234567$.
Each vertex $v$ of the tree represents a subset $b$ of $\{1,\dots,n\} $ of cardinality at least $2$,
that appears as a block of an element in the chain. This vertex $v$ has label $i$ if the block $b$
appears after the $i$th merging. If $v_1,v_2$ are two vertices and $b_1,b_2$ the corresponding subsets
of $\{1,\dots,n\}$ then $v_1$ is below $v_2$ in the tree if $b_1\subset b_2$.
In the example of Figure~\ref{stanmap}, the correspondence between blocks and labels is:
$46 \to 1$, $15 \to 2$, $37 \to 3$, $3467 \to 4$, $125 \to 5$ , $1234567 \to 6$.

\begin{figure}[h!tp]
\parbox{4cm}{\small $1234567$ \\ $125|3467$ \\ $15|2|3467$   \\ $15|2|37|46$
            \\ $15|2|3|46|7$ \\ $1|2|3|46|5|7$ \\ $1|2|3|4|5|6|7$  }
\hspace{1cm}
\begin{pspicture}(0,0)(0,-0.7)
\pstree[levelsep=5mm,treesep=3mm]
       { \Tcircle{\tiny 6} }
       { 
          {    \pstree[levelsep=6mm,treesep=8mm]  { \Tcircle{\tiny 5} } {   { \Tcircle{\tiny 2} }   }
          }
          {\pstree[levelsep=6mm,treesep=3mm]  { \Tcircle{\tiny 4} }  { 
                                                           { \Tcircle{\tiny 3} }{ \Tcircle{\tiny 1} }
                                                        } 
          }
       }
\end{pspicture}
\caption{Stanley's bijection. \label{stanmap} } 
\end{figure}
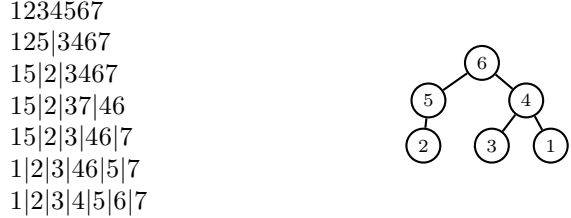

\begin{prop}
Let $\Pi\in\mathcal{M}^{NC}(A_{n-1})$, and $T\in\mathcal{A}_n$ its image under Stanley's bijection. Then we have
\[
  M([\Pi],q) = \prod_{\substack{ v \in T  \\ h_v \neq 1 } } ( 2 + q(h_v-1) ).
\]
where $h_v$ is the hook of the vertex $v$.
\end{prop}

\begin{proof}
Let $2\leq i\leq n$. There are $a>0$ and $b>0$ such that $\pi_i$ is obtained from $\pi_{i-1}$ by merging
two blocks of size $a$ and $b$ into one block of size $a+b$. The integer $h_i$ is the Coxeter number 
of $\mathfrak{S}_{a+b}$, i.e. $h_i=a+b$. If $a>1$ or $b>1$, i.e. one of the two blocks has cardinality at least 2, 
there is a nontrivial factor $\mathfrak{S}_a$ or $\mathfrak{S}_b$ that needs $a+b$ rotations through the cycle to 
go back to itself, so that $g_i=a+b$. But if $a=b=1$, we have $g_i=1=\frac{h_i}2$.

Let $v$ be the vertex of $T$ with label $i$. From the properties of the bijection, 
the two sons of $v$ contains $a-1$ and $b-1$ vertices, and $h_v=a+b-1$.
So, we obtain:
\[
  \frac{2g_i}{h_i} + q(g_i - \frac{2g_i}{h_i} ) =
  \begin{cases} 2 + q(h_v-1) \text{ if } h_v>1, \\
                1 \text{ otherwise.}
  \end{cases}
\]
So Proposition~\ref{classgen} specializes as stated above.
\end{proof}

As a consequence, Equation~\eqref{classeq} gives the following:

\begin{theo}
\begin{equation} \label{hookA}
  \prod_{i=1}^{n-1} ( i+1 + q(n-i) )  =  \sum_{T \in \mathcal{A}_n } \prod_{\substack{ v \in T  \\ h_v \neq 1 } } ( 2 + q(h_v-1) ). 
\end{equation}
\end{theo}

For example, for $n=4$, and taking the 5 trees as in Figure~\ref{tree5}, we get:
\begin{align*}
  (2+3q)(3+2q)(4+q) & = (2+q)(2+2q)(2+3q) + (2+2q)(2+3q) +  \\
                    & \qquad (2+q)(2+3q) + (2+q)(2+3q)+(2+q)(2+3q).
\end{align*}

We have to make the connection with previously-known results.
Let $\mathcal{T}_n$ denote the set of binary plane trees on $n$ vertices, 
and $\mathcal{T}^\ell_n$ denote the set of pairs $(T,L)$ where $T\in\mathcal{T}_n$
and $L$ is a decreasing labeling of the vertices. It is well-known that the number 
such labelings $L$ for a given $T$ is 
\[
   \frac{n!}{ \prod_{v\in T} h_v }.
\]
Moreover, there is a map $\mathcal{T}^\ell_n \to \mathcal{A}_n $ which consists in ``forgetting''
the notion of left and right among the sons of each internal vertex. It is such that each $T \in \mathcal{A}_n $ 
has $2^{{\rm in}(T)}$ preimages, where ${\rm in}(T)$ is the number of internal vertices of $T$ 
(i.e. $v\in T$ such that $h_v>1$).
Then, we can rewrite the right-hand side of \eqref{hookA}:
\begin{align*}
    & \sum_{T \in \mathcal{A}_n } \prod_{\substack{ v \in T  \\ h_v \neq 1 } } ( 2 + q(h_v-1) )
  =   \frac{1}{2^n} \sum_{T \in \mathcal{A}_n }  2^{{\rm in}(T)} \prod_{ v \in T } ( 2 + q(h_v-1) ) \\
  & = \frac{1}{2^n} \sum_{T \in \mathcal{T}^\ell_n } \prod_{ v \in T } ( 2 + q(h_v-1) ) 
  =   \frac{n!}{2^n} \sum_{T \in \mathcal{T}_n } \prod_{ v \in T } \big( \frac{ 2 + q(h_v-1) } {h_v} \big).
\end{align*}

So we arrive at
\[
  \prod_{i=1}^{n-1} ( i+1 + q(n-i) ) = \frac{n!}{2^n} \sum_{T \in \mathcal{T}_n } \prod_{ v \in T } ( q + \frac{2-q}{h_v} ).
\]
The particular case $q=1$ is Postnikov's hook-length formula \cite[Corollary 17.3]{postnikov},
proved in investigating the volume of generalized permutohedra.
A one-parameter generalization was conjectured by Lascoux and proved by Du and Liu \cite{du},
it is exactly the previous equation up to the change of variable $(q,2-q)\to(q,1)$.

Let us turn to the type B analogue, where we can adapt Stanley's bijection.
(Note that a type B analogue of André trees or permutations has been considered by Purtill \cite{purtill}, in 
relation with type B Springer numbers.)

For brevity, the integers $-1$, $-2$, etc. will be represented $\bar 1$, $\bar 2$, etc.
A set partition of type B is a set partition of $\{\bar n,\dots, \bar 1\}\cup \{1,\dots,n\}$, unchanged under 
the map $x\to -x$, and such that there is at most one block $b$ such that $b=-b$ (called the 0-block when it exists).
For example, $1 \bar 2 5|\bar 1 2 \bar 5 | 3 \bar 3 6 \bar 6 | 4 | \bar 4 \in \mathcal{P}(B_6)$.

\begin{defi}
  A {\it pointed André tree} is an André tree with a distinguished vertex $v\in T$ having 0 or 1 son.
  Let $\mathcal{A}^*_n$ denote the set of pointed André trees on $n$ vertices.
\end{defi}

A tree $T\in \mathcal{A}^*_n$ is represented with the convention that the distinguished vertex has a starred label $i^*$.
We can create a new tree as follows: increase the labels by 1, then add a new vertex with label $1$ attached to the 
distinguished vertex. This is clearly a bijection between $\mathcal{A}^*_n$ and $\mathcal{A}_{n+1}$, showing that 
$\# \mathcal{A}^*_n = T_{n+1} = K(B_n) $. See Figure~\ref{treebij} for an example.

\begin{figure}[h!tp]
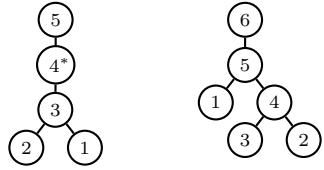

\pstree[levelsep=6mm]
       { \Tcircle{\tiny 5} }
       { \pstree[levelsep=6mm,treesep=1mm]
            {   \Tcircle{\tiny $4^*$ \hspace{-3mm} } }
            {{  \pstree[levelsep=5mm,treesep=3mm]
            { \Tcircle{\tiny 3} }
            {{ \Tcircle{\tiny 2} }{ \Tcircle{\tiny 1}}}}}
       }
\hspace{1cm}  
\pstree[levelsep=6mm]
       { \Tcircle{\tiny 6} }
       { \pstree[levelsep=5mm,treesep=3mm]
            {   \Tcircle{\tiny 5} }
            {{ \Tcircle{\tiny 1} }{  \pstree[levelsep=5mm,treesep=3mm]
            { \Tcircle{\tiny 4} }
            {{ \Tcircle{\tiny 3} }{ \Tcircle{\tiny 2}}}}}
       }
\caption{ The bijection $\mathcal{A}^*_n \to \mathcal{A}_{n+1} $. \label{treebij} } 
\end{figure}

Let $\Pi=(\pi_0,\dots,\pi_n)\in\mathcal{M}(B_n)$. We build a tree $T\in \mathcal{A}_n^*$ by adapting Stanley's map. 
A vertex in $T$ represents either the 0-block in some $\pi_i$, or a pair of distinct opposite blocks in some $\pi_i$
where the elements of the pair have cardinality at least 2.
This vertex has label $i$ if this 0-block, or pair of opposite blocks, appears in $\pi_i$ but not in $\pi_{i-1}$.
A vertex $v_1$ is below another vertex $v_2$ in the tree when the blocks represented by $v_1$ are included in the blocks 
represented by $v_2$.
Eventually, we have the following rule: the distinguished vertex has label $i$ if and only if 
$\pi_{i}$ has a 0-block, and $\pi_{i-1}$ has none.
See Figure~\ref{stanleyB} for an example.

\begin{figure}[h!tp]
\parbox{4cm}{\small $1\bar 1 2 \bar 2 3 \bar 3 4 \bar 4 5 \bar 5 6 \bar 6 $ \\ $1\bar 13 \bar 3 |2\bar 456|\bar 24 \bar 5 \bar 6$ 
              \\ $1\bar 13 \bar 3 |25|\bar 2\bar 5 | 4\bar 6 |\bar 46$  \\ $13 | \bar 1 \bar 3 |25|\bar 2\bar 5 | 4\bar 6 |\bar 46$ 
              \\ $1  | \bar 1 | 3 | \bar 3 | 2 5 |\bar 2\bar 5 | 4\bar 6 |\bar 46$ 
              \\ $1  | \bar 1 | 2 | \bar 2 | 3  |\bar 3 | 5 | \bar 5 | 4\bar 6 |\bar 46$ 
              \\ $1  | \bar 1 | 2 | \bar 2 | 3 | \bar 3 | 4 | \bar 4 | 5 | \bar 5 | 6 | \bar 6$ }
\hspace{1cm}
\begin{pspicture}(0,0)(0,-0.7)
\pstree[levelsep=5mm,treesep=3mm]
       { \Tcircle{\tiny 6} }
       { 
          {\pstree[levelsep=6mm,treesep=3mm]  { \Tcircle{\tiny 5} }  { 
                                                           { \Tcircle{\tiny 2} }{ \Tcircle{\tiny 1} }
                                                        } 
          }
          {    \pstree[levelsep=6mm,treesep=8mm]  { \Tcircle{\tiny $4^*$ \hspace{-3mm} } } {   { \Tcircle{\tiny 3} }   }
          }
       }
\end{pspicture}
\caption{Stanley's bijection adapted to type B. \label{stanleyB} } 
\end{figure}
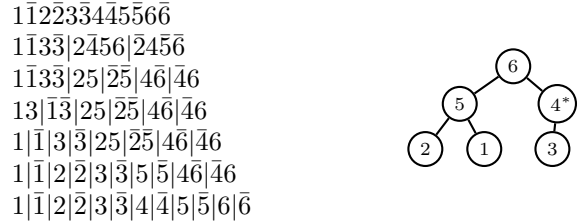

\begin{prop}
Let $\Pi\in\mathcal{M}(B_n)$ and $T\in\mathcal{A}^*_n$ its image under the bijection we have just defined.
For any vertex $v$ of the tree $T\in\mathcal{A}^*_n$, we define a factor $\beta(v)$ to be $1+q(h_v-1)$ if $v$ 
belongs to the minimal path joining the root to the distinguished vertex, $2+q(h_v-1)$ otherwise. Then we have:
\[
  M([\Pi],q) = \prod_{\substack{ v\in T \\ h_v \neq 1 }} \beta(v).
\]
\end{prop}

\begin{proof}
Let $2\leq i\leq n$, let $v$ be the vertex with label $i$.

Suppose first that $\pi_i$ is obtained from $\pi_{i-1}$ by merging two pairs of distinct opposite blocks into
a pair of distinct opposite blocks
(such as $25|\bar 2 \bar 5$ and $4 \bar 6|\bar4 6 $ in the example). This is the case where $v$ is not in the 
minimal path from the root to the distinguished vertex.
This means that $W_{(\pi_i)}$ is obtained from $W_{(\pi_{i-1})}$ by replacing a factor $\mathfrak{S}_a\times \mathfrak{S}_b$
with $\mathfrak{S}_{a+b}$. As in the type A case, we get $g_i=h_i=a+b+1$, and $a-1$, $b-1$ are the number of vertices in
the subtrees of $v$. This gives $\frac{2g_i}{h_i} + q(g_i - \frac{2g_i}{h_i} ) = \beta(v)$.

Suppose then that $\pi_i$ is obtained from $\pi_{i-1}$ by merging two pairs of distinct opposite blocks into a 
0-block (such as $13$ and $\bar 1\bar 3$ in the example). This is the case where $v$ is the distinguished vertex.
This means that $W_{(\pi_i)}$ is obtained from $W_{(\pi_{i-1})}$ by replacing a factor $\mathfrak{S}_j = A_{j-1}$
into $B_j$ where $j$ is the size of the 0-block, and also the hook-length of $v$. We obtain $h_i=2j$, and $g_i=j$.
Also in this case, this gives $\frac{2g_i}{h_i} + q(g_i - \frac{2g_i}{h_i} ) = \beta(v)$.

Eventually, suppose that $\pi_i$ is obtained from $\pi_{i-1}$ by merging a pair of distinct opposite blocks to the
0-block (such as $2\bar 456|\bar 2 4 \bar 5 \bar 6$ in the example). This is the case where $v$ is in the minimal 
path from the root to the distinguished vertex (but is not the distinguished vertex).
This means that $W_{(\pi_i)}$ is obtained from $W_{(\pi_{i-1})}$ by replacing a factor $A_{j-1} \times B_k $
into $B_{j+k}$. Here, $k>0$ is the number of vertices in the subtree of $v$ containing the distinguished vertex,
and $j-1\geq 0$ is the number of vertices in the other subtree.
We get $h_i=2(j+k)$,  $g_i=j+k=h_v$,
and $\frac{2g_i}{h_i} + q(g_i - \frac{2g_i}{h_i} ) = \beta(v)$. 

So Proposition~\ref{classgen} specializes as stated above.
\end{proof}

So, in the type B case, Equation~\eqref{classeq} gives:

\begin{theo}
\[
  \prod_{i=1}^{n} ( i + q(n-i) )  =  \sum_{T \in \mathcal{A}^*_n } \prod_{\substack{ v \in T \\ h_v \neq 1 } } \beta(v).
\]
\end{theo}

For example, let $n=3$. We take the 5 elements of $\mathcal{A}^*_n$ as they appear in Figure~\ref{tree5} after we apply the
bijection $\mathcal{A}_{n+1}\to \mathcal{A}^*_n$, and we get:
\begin{align*}
  3(2+q)(1+2q) & = (1+q)(1+2q) + (1+q)(1+2q) + (1+2q) + \\ 
               & \qquad (1+2q) + (2+q)(1+2q).
\end{align*}

Strictly speaking, the identity in the previous theorem might be not considered as a hook-length formula since $\beta(v)$ 
does not depend only on the hook-length $h_v$. Still, it is on its own an interesting variant of the type A case.

\appendix

\section{Properties of the standard Coxeter elements}

\label{appen}

We sketch here a case-by-case proof of Propositions~\ref{standard1}. % and \ref{standard2}.
As we have seen above, the result is proved in the case where the longest element is central.
It remains only to prove the result for the infinite families $A_{n-1}$, $D_{n}$, 
and for the exceptional group $E_6$.
% In types A, B, D, the results become clear upon inspection, once we have explicit combinatorial 
% descriptions of the objects. As for the exceptional groups, everything can be checked with a computer program.

We shall use the notion of cyclic order and cyclic intervals. 
Recall that a sequence $i_1,\dots,i_n$ is {\it unimodal} if there is $k$ such that $i_1\leq i_2\leq\dots\leq i_k$ and 
$i_k\geq i_{k-1}\geq \dots \geq i_n$.

\subsection{\texorpdfstring{Case of $A_{n-1}$}{Case of A n-1}}

Let $W=A_{n-1}=\mathfrak{S}_n$, $V=\{v\in\mathbb{R}^{n} \, : \, \sum v_i = 0 \}$.
Let $S=\{s_1,\dots,s_{n-1}\}$, where $s_i$ acts by permuting the $i$th and $(i+1)$th
coordinates. As a permutation, $s_i$ is the simple transposition $(i,i+1)$.
Let $c=s_{\sigma(1)}\dots s_{\sigma(n-1)} $ be a standard Coxeter element.
By exchanging pairs of commuting generators, we can write $c$ as a product of $s_{n-1}$ with a standard Coxeter element 
of $A_{n-2}$ (where we do not specify the order of the product). By an easy induction, we see that we can write $c$ as 
the cycle $(i_1,\dots,i_{n})$ where $i_1,\dots,i_n$ is a unimodal sequence (and a permutation of $1,\dots,n$).

Any coatom of $\mathcal{P}^{NC}(A_{n-1},c)$ is a pair of cyclic intervals of the sequence $i_1,\dots,i_n$, 
complementary to each other, and the action of $c$ is the ``rotation'' along the cycle.
Two such coatoms are in the same $c$-orbit if and only if they have the same block sizes.
So, for each $k$ with $1\leq k < \frac n2$, there is an orbit containing 
complementary cyclic intervals of size $k$ and $n-k$. There are $n$ such partitions, and the interval
partitions among them are $1\dots k|k+1\dots n$ and $1\dots n-k|n-k+1\dots n$.
Additionally, if $n$ is even, there is an orbit containing two complementary cyclic intervals of size $\frac n2$.
There are $\frac n2$ such partitions, and the only interval partition among them is $1\dots \frac n2|\frac n2+1\dots n$.
Proposition~\ref{standard1} follows.

% We turn to Proposition~\ref{standard2}. As simple roots of $W$, we take $e_i-e_{i+1}$ for $1\leq i \leq n-1$ where 
% $(e_i)_{1\leq i\leq n+1}$ is the canonical basis of $\mathbb{R}^{n}$.
% Let $t$ be a reflection, then there are $1\leq \ell < m \leq n$, such that $t$ is the transposition $(i_{\ell},i_m)$. 
% The permutation $ct$ is a product of two cycles:
% \[
%   ct = (i_1,\dots,i_{\ell},i_{m+1},\dots,i_n) (i_{\ell+1},\dots,i_m ).
% \]
% The two sequences $i_1,\dots,i_{\ell},i_{m+1},\dots,i_n$ and $i_{\ell+1},\dots,i_m$ are also unimodal, as subsequences
% of a unimodal sequence. 
% The group $W_{(ct)}$ is a product of two symmetric groups, one acting on $i_1,\dots,i_{\ell},i_{m+1},\dots,i_n$ and the 
% other on $i_{\ell+1},\dots,i_m$.
% Its positive roots are $e_u-e_v$ where $u<v$ are two indices in one of the sequence 
% $1,\dots,\ell,m+1,\dots,n$ or $\ell+1,\dots,m$, and taking two consecutive indices give the simple roots.
% As in the general case seen above, we can deduce that a standard Coxeter element of $W_{(ct)}$ is a product of 
% two cycles given by unimodal sequences, i.e. exactly what we have obtained for $ct$.

\subsection{\texorpdfstring{Case of $B_n$}{Case of B n}}

Proposition~\ref{standard1} was already proved in this case, since the longest element is central.
So the goal of this section is only to introduce some notation nedeed in the type D case (because
we see $D_n$ as a subgroup of $B_n$ in the standard way).
% We turn directly to Proposition~\ref{standard2}. 
Let $W=B_n$ acting on $V=\mathbb{R}^n$.
The group $B_n$ is generated by $s_1,\dots,s_{n-1}$, i.e. generators of $A_{n-1}$,
together with another generator $s_0^B$. The latter acts as $(v_1,\dots,v_n)\mapsto (-v_1,v_2,\dots,v_n)$.
The simple roots are $-e_1$, together with $e_i-e_{i+1}$ for $1\leq i < n$.
We identify $B_n$ with the group of signed permutations, and $s_0^B$ is the transposition $(1,-1)$.
We use the notation $((a_1,\dots,a_n)) = (a_1,\dots,a_n)(-a_1,\dots,-a_n)$ and $[[a_1,\dots,a_n]] = (a_1,\dots,a_n,-a_1,\dots,-a_n)$
for the cycles of signed permutations.

\subsection{\texorpdfstring{Case of $D_n$}{Case of D n}}

The group $D_n$ is the subgroup of $B_n$ generated by $s_1,\dots,s_{n-1}$ together with another generator $s_0^D$.
The latter acts by the transformation
\[ v=(v_1,\dots,v_n) \mapsto (-v_2, -v_1, v_3, \dots, v_n).  \]
As a signed permutation, 
it is the transposition $((-1,2))$. The simple roots are $-e_1-e_2$, and $e_i-e_{i+1}$ for $1\leq i <n$.
% Note that this is the natural choice induced by our previous choice of positive roots for $B_n$ when we see $D_n$
% as a subgroup of $B_n$.
By exchanging pairs of commuting generators, we can see that a standard Coxeter element $c$ of $D_n$ 
is a product of $s_0^D$ and a standard Coxeter element of $A_{n-1}$. So, either:
\[
  c = (1,-1) [[i_1,\dots,i_{n-1} ]]
\]
where $i_1,\dots,i_{n-1}$ form a unimodal sequence, and a permutation of $2,\dots,n$, or:
\[
  c = (2,-2) [[i_1,\dots,i_{n-1} ]]
\]
where $i_1,\dots,i_{n-1}$ form a unimodal sequence, and a permutation of $1,3,\dots,n$.
We only consider the first case, the other one being completely similar (it suffices to replace
each 1 with a 2 in the text).

We have four kinds of products $ct$ where $t$ is a reflection:
\begin{align*}
  c ((1,i_m)) &= ((1,i_{m+1},\dots,i_{n-1},-i_1,\dots,-i_m)), \\
  c ((-1,i_m)) &=  ((1,-i_{m+1},\dots,-i_{n-1},i_1,\dots,i_m)), \\
  c ((i_\ell,i_m)) &= (1,-1) [[i_1,\dots,i_{\ell},i_{m+1},\dots,i_{n-1}]]((i_{\ell+1},\dots,i_m )), \\
  c ((-i_\ell,i_m)) &= (1,-1)  [[i_{\ell+1},\dots,i_m]]  ((i_1,\dots,i_{\ell},-i_{m+1},\dots,-i_{n-1})).
\end{align*}

Using the notation for type B set partitions, we obtain from the list
above that the coatoms of $\mathcal{P}^{NC}(D_n,c)$ are:
\begin{itemize}
 \item $ 1 i_{m+1}\dots i_{n-1} \bar i_1 \dots \bar i_m  | \bar 1 \bar i_{m+1} \dots \bar i_{n-1} i_1 \dots i_m  $,
 \item $ \bar 1 i_{m+1}\dots i_{n-1} \bar i_1 \dots \bar i_m  |  1 \bar i_{m+1} \dots \bar i_{n-1} i_1 \dots i_m  $,
 \item $ 1 i_1\dots i_\ell i_{m+1} \dots i_{n-1} \bar 1 \bar i_1\dots \bar i_\ell \bar i_{m+1} \dots \bar i_{n-1} 
        |  i_{\ell+1}\dots i_m | \bar i_{\ell+1} \dots \bar i_m $,
 \item $ 1 i_{\ell+1}\dots i_m \bar 1 \bar i_{\ell+1} \dots \bar i_m | 
           i_1\dots i_\ell \bar i_{m+1} \dots \bar i_{n-1} | \bar i_1 \dots \bar i_\ell i_{m+1} \dots i_{n-1} $.
\end{itemize}
And the interval partitions among them are $1\dots n|\bar 1 \dots \bar n$, $1 \bar 2\dots \bar n|\bar 1 2 \dots n$, and
\[
  1\dots i \bar 1 \dots \bar i | i+1 \dots n  | \overline{i+1} \dots \bar n,
\]
where $2\leq i <n$.
From these explicit description, we can check Proposition~\ref{standard1}. We find that 
all orbits have size $\frac h2$ (here $h=2n-2$), except that 
$1\dots n|\bar 1 \dots \bar n$ and $1 \bar 2\dots \bar n|\bar 1 2 \dots n$ are in a same orbit of size $h$
if $n$ is even.

% \begin{align*}
%   c'' ((2,i_m)) &= ((2,i_{m+1},\dots,i_{n-1},-i_1,\dots,-i_m)), \\
%   c'' ((-2,i_m)) &=  ((2,-i_{m+1},\dots,-i_{n-1},i_1,\dots,i_m)), \\
%   c'' ((i_\ell,i_m)) &= (2,-2) [[i_1,\dots,i_{\ell},i_{m+1},\dots,i_{n-1}]]((i_{\ell+1},\dots,i_m )), \\
%   c'' ((-i_\ell,i_m)) &= (2,-2)  [[i_{\ell+1},\dots,i_m]]  ((i_1,\dots,i_{\ell},-i_{m+1},\dots,-i_{n-1})).
% \end{align*}

% It remains to check Proposition~\ref{standard2}.
% For the first and second kind of products $ct$, we can use Lemma~\ref{ansubbn} to find that it is a standard 
% Coxeter element for a subgroup of type $A_{n-1}$.
% For the third kind, we can directly recognize a standard Coxeter element of type $D_{n-m+\ell} \times A_{m-\ell-1}$.
% For the fourth kind, we can recognize a standard Coxeter element of type $D_{m-\ell} \times A_{n-m+\ell-1}$
% (using Lemma~\ref{ansubbn} for the second factor).

\subsection{\texorpdfstring{Case of $E_6$}{Case of E 6}}

%The only case of Proposition~\ref{standard1} that remains to be checked is the one of $E_6$.
This can be done with the following Sage program \cite{sage} (tested with Sage 5.4).

{\small
\begin{verbatim} 
W = WeylGroup(['E',6])
n = 6
h = 12

S = W.simple_reflections()
w0 = W.long_element()

def checkorbits(l):
  c = prod( S[i] for i in l )
  inte = []
  for i in range(1,n+1):
    inte.append( prod( S[j] for j in l if j!=i ) )
  for ct in inte:
    i=1; j=1; k= c * ct * c**(-1) ;
    while k != ct :
      i+=1
      if k in inte: 
        j+=1
        ct2 = k
      k = c * k * c**(-1)
    if not (((j==2) and (i==h)) or ((mod(h,2)==0) and (i==h/2) and (j==1))):
      raise TypeError('ERROR!!!') 
    if not (((j==2) and (ct2==w0*ct*w0)) or ((j==1) and (ct == w0*ct*w0))):
      raise TypeError('ERROR!!!')

for l in Permutations(n):
  checkorbits(l)
\end{verbatim}
}

\setlength{\parindent}{0mm}

\end{document}